\documentclass[a4paper,12pt]{article}
\usepackage{amsmath}
\usepackage{amssymb}
\usepackage{amsthm}

\theoremstyle{plain}
\newtheorem{theorem}{Theorem}[section]

\newtheorem{lemma}{Lemma}[section]
\newtheorem{corollary}{Corollary}[section]
\theoremstyle{definition}
\newtheorem{definition}{Definition}[section]

\begin{document}
	
	\title{The Dirichlet series of the arithmetic derivative}
	
	\author{
		{ \sc Es-said En-naoui } \\ 
		Faculty Of Science And Technics,University Sultan Moulay Slimane\\ Morocco\\
		essaidennaoui1@gmail.com\\
		\\
	}
	
	\maketitle

	\begin{abstract}
	The main object of this paper is to give the generalized von mangoldt function using the L-additive function which can help us to make it possible to calculate The Dirichlet series of the arithmetic derivative $\delta$ and Dirichlet series defined by:
	 $$
	\sum \limits_{n\geq 1} \frac{f(n)\delta(n)}{n^s}
	$$ where $f$ is a classical arithmetic function.
	\end{abstract}
	
	\section{Introduction}
	The arithmetic derivative of a natural number $n$, denoted by $\delta(n)$ or $D(n)$ or $n'$, has been the subject of extensive study, from Barbeau, E. J (see, e.g., \cite{Ba}) to P. Haukkanen, J. K. Merikoski, and T. Tossavainen (see, e.g., \cite{Ufna}) .\\
	
	In this work, we address two major result: Dirichlet product of L-additive function and generalized von Mangoldt function using L-additive function $f$.\\
In this article, we use the generalized function von Mangoldt using the L-additive function to find alternative proof for many series expansions depend on the arithmetic derivative function.
	The methods readily generalize, and can be applied to other L-additive functions. Our  principal result are that :
	$$
	\sum \limits_{n\geq 1} \frac{\delta(n)}{n^s}=\zeta(s-1)\sum \limits_{p}\frac{1}{p^{s}-p}
	$$	where $\delta$ is the arithmetic derivative function .
	
\medskip
\noindent
	 Let $n$ be a positive integer. Its {\it arithmetic derivative}
	is the
	function $\delta\;:\; \mathbb N \rightarrow \mathbb N$ ,$\;$defined by the rules : 
	
	\begin{enumerate}
		\item  $\delta(p)=1$ for all primes $p$
		\item  $\delta(mn)=m\delta(n)+n\delta(m)$ for all positive integers $m$ and $n$ (the Leibnitz rule)
	\end{enumerate}
	
	\medskip
	\noindent
	
	Let $n$ a positive integer , if $n=\prod_{i=1}^{s}p_{i}^{\alpha_{i}}$ is the prime factorization of $n$, then
	the formula for computing the arithmetic derivative
	of n is (see, e.g., \cite{Ba, UA}) giving by :
	\begin{equation}
		\delta(n)=n\sum \limits_{i=1}^s\frac{\alpha_i}{p_i}
		=n\sum \limits_{p^{\alpha}||n}\frac{\alpha}{p}
	\end{equation}
	A brief summary on the history of arithmetic derivative and its generalizations to other number sets can be found, e.g., in \cite{Ba, UA, HMT}. 
	
	Similarly, one can define {\it the arithmetic logarithmic derivative} \cite{UA} as
	$$
	{\rm ld}(n)=\frac{\delta(n)}{n}.
	$$
	
	First of all, to cultivate analytic number theory one must acquire a considerable 
	skill for operating with arithmetic functions. We begin with a few elementary  considerations. 
	
	\begin{definition}[arithmetic function]
		An \textbf{arithmetic function} is a function $f:\mathbb{N}\longrightarrow \mathbb{C}$ with
		domain of definition the set of natural numbers $\mathbb{N}$ and range a subset of the set of complex numbers $\mathbb{C}$.
	\end{definition}
	
	\begin{definition}[multiplicative function]
		A function $f$ is called an \textbf{multiplicative function} if and
		only if : 
		\begin{equation}\label{eq:1}
			f(nm)=f(n)f(m)
		\end{equation}
		for every pair of coprime integers $n$,$m$. In case (\ref{eq:1}) is satisfied for every pair of integers $n$ and $m$ , which are not necessarily coprime, then the function $f$ is
		called \textbf{completely multiplicative}.
	\end{definition}
	
	Clearly , if $f$ are a multicative function , then
	$f(n)=f(p_1^{\alpha _1})\ldots f(p_s^{\alpha _s})$, 
	for any positive integer $n$ such that 
	$n = p_1^{\alpha _1}\ldots  p_s^{\alpha _s}$ , and if $f$ is completely multiplicative , so we have : $f(n)=f(p_1)^{\alpha _1}\ldots f(p_s)^{\alpha _s}$.
	
	The functions defined above are widely studied in the literature, (see, e.g., \cite{KM,LT,Mc,SC,Sc, Sh}). 
	
	\bigskip
	\begin{definition}[additive function]
		A function $f$ is called an \textbf{additive function} if and
		only if : 
		\begin{equation}
			f(nm)=f(n)+f(m)
		\end{equation}
		for every pair of coprime integers $n$,$m$. In case (3) is satisfied for every pair of integers $n$ and $m$ , which are not necessarily coprime, then the function $f$ is
		called \textbf{completely additive}.
	\end{definition}
	Clearly , if $f$ are a additive function , then
	$f(n)=f(p_1^{\alpha _1})+\ldots +f(p_s^{\alpha _s})$, 
	for any positive integer $n$ such that 
	$n = p_1^{\alpha _1}\ldots  p_s^{\alpha _s}$ , and if $f$ is completely additive , so we have : $f(n)=f(p_1)^{\alpha _1}+\ldots +f(p_s)^{\alpha _s}$.
	\bigskip
	\begin{definition}[L-additive function]
		We say that an arithmetic function $f$ is {\em Leibniz-additive} (or, {\em L-additive}, in short) (see, e.g., \cite{Ufna}) if there is a completely multiplicative function $h_f$ such that 
		\begin{equation}\label{gca}
			f(mn)=f(m)h_f(n)+f(n)h_f(m)
		\end{equation}
		for all positive integers $m$ and $n$. 
	\end{definition}
	Then $f(1)=0$ since $h_f(1)=1$. 
	The property \eqref{gca} may be considered a generalized Leibniz rule.
	For example, the arithmetic derivative $\delta$ is L-additive with $h_\delta(n)=n$, 
	since it satisfies the usual Leibniz rule 
	$$
	\delta(mn)=n\delta(m)+m\delta(n)
	$$
	for all positive integers $m$ and $n$, and the function $h_\delta(n)=n$ is completely multiplicative. Similarly, the arithmetic partial derivative respect to the prime $p$ is L-additive with $h_{\delta_p}(n)=n$.
	Further, all completely additive functions $f$ are L-additive with $h_f(n)=1$. For example, the logarithmic derivative of $n$ is completely additive since
	$$
	{\rm ld}(mn) = {\rm ld}(m)+{\rm ld}(n).
	$$

	The term ``L-additive function" seems to be new in the literature, yet Chawla \cite{Ch} has defined  the concept of completely distributive arithmetic function meaning the same as we do with an L-additive function. However, this is a somewhat misleading term since a distributive arithmetic function usually refers to a property that 
	\begin{equation}\label{distr}
		f(u\ast v)=(fu)\ast(fv),
	\end{equation}
	i.e., the function $f$ distributes over the Dirichlet convolution. This is satisfied by completely multiplicative arithmetic functions, not by completely distributive functions as Chawla defined them. 
	
	Because L-additivity is analogous with generalized additivity and generalized multiplicativity (defined in \cite{Ha}), we could, alternatively, speak about generalized complete additivity (and also define the concept of generalized complete multiplicativity). 
	
	In this paper, we consider L-additive functions especially from the viewpoint that they are generalizations of the arithmetic derivative. In the next section, we present their basic properties. In the last section, we study L-additivity and the arithmetic derivative in terms of the Dirichlet convolution.
	
	
	\begin{theorem}\label{the-1-1}
		Let $f$ be an arithmetic function. If $f$ is L-additive and
		$h_f$ is nonzero-valued, then $f/h_f$ is completely additive.
	\end{theorem}
	
	\begin{proof} 
		If $f$ satisfies~\eqref{gca} and $h_f$ is never zero, then
		$$
		\frac{f(mn)}{h_f(mn)}=\frac{f(m)h_f(n)+f(n)h_f(m)}{h_f(m)h_f(n)}=
		\frac{f(m)}{h_f(m)}+\frac{f(n)}{h_f(n)},
		$$
		
	\end{proof}

	\begin{theorem}\label{totally}
		Let $n$ a positive integer , if $n=\prod_{i=1}^{s}p_{i}^{\alpha_{i}}$ is the prime factorization of $n$ and $f$ is L-additive with $h_f(p_1),\dots,h_f(p_s)\ne 0$, then 
		$$
		f(n)=h_f(n)\sum_{i=1}^s\frac{\alpha_i f(p_i)}{h_f(p_i)}. 
		$$		
	\end{theorem}
	\begin{proof}
		(see, e.g., \cite[Theorem 2.4]{Ufna})
	\end{proof}
	\medskip

	The next step is to extend the L-additive function to the set of rational numbers $\mathbb{Q}^*$ to us it in the Direchlet product of L-additive function , We start from the positive rationals.\\
	
	The shortest way is to use the theorem $1.1$ . Namely, if $x=\prod_{i=1}^{s}p_{i}^{x_{i}}$ is a factorization of a rational
	number $x$ in prime powers, (where some $x_i$ may be negative) then we put :
	\begin{equation}
		f(x)=h_f(x)\sum_{i=1}^s\frac{x_i f(p_i)}{h_f(p_i)}.
	\end{equation}
	and the same proof as in Theorem $1.1$ shows that this definition is still consistent with the Leibnitz rule for every L-additive function $f$ with $h_f\neq0$.
	\begin{lemma}\label{inverse}
		Let $n$ a positive integer , if $n=\prod_{i=1}^{s}p_{i}^{\alpha_{i}}$ is the prime factorization of $n$, then and $f$ is L-additive function with
		$h_f$ is nonzero-valued, then :
		\begin{equation}
			f\bigg(\frac{1}{n}\bigg)=\frac{-f(n)}{h^2_f(n)}
		\end{equation}
	\end{lemma}
	\begin{proof} 
		Let $n$ a positive integer , if $n=\prod_{i=1}^{s}p_{i}^{\alpha_{i}}$ is the prime factorization of $n$ , then we have by the formula $(6)$ : 
		$$
		f\bigg(\frac{1}{n}\bigg)=h_f\bigg(\frac{1}{n}\bigg)\sum_{i=1}^s\frac{-\alpha_i f(p_i)}{h_f(p_i)}=-h_f\bigg(\frac{1}{n}\bigg)\sum_{i=1}^s\frac{\alpha_i f(p_i)}{h_f(p_i)}
		$$
		Since $f(n)=h_f(n)\sum_{i=1}^s\frac{\alpha_i f(p_i)}{h_f(p_i)} $ then , $\sum_{i=1}^s\frac{\alpha_i f(p_i)}{h_f(p_i)}=\frac{f(n)}{h_f(n)}$ , so we have :
		$$f\bigg(\frac{1}{n}\bigg)=-h_f\bigg(\frac{1}{n}\bigg).\frac{f(n)}{h_f(n)}
		=-h_f(n).h_f\bigg(\frac{1}{n}\bigg).\frac{f(n)}{h_f^2(n)}=\frac{-f(n)}{h^2_f(n)}
		$$
		because $h_f$ is multiplicative and $h_f(n).h_f\bigg(\frac{1}{n}\bigg)=h_f\bigg(\frac{n}{n}\bigg)=h_f(1)=1$ .
	\end{proof}
	
	\begin{theorem}\label{quotient}
		Lets $n$ and $m$ two positive integers with $m\neq0$. and $f$ is L-additive function with
		$h_f$ is nonzero-valued, then we have :
		\begin{equation}
			f\bigg(\frac{n}{m}\bigg)=\frac{f(n)h_f(m)-f(m)h_f(n)}{h^2_f(m)}
		\end{equation}
		A L-additive can be well defined for rational numbers using this formula and this is the only
		way to define a L-additive over rationals that preserves the Leibnitz rule.
	\end{theorem}
	
	\begin{proof} 
		If $n$ and $m$ two positive integers with $m\neq0$ and $h_f$ is never zero, then :
		$$
		f\bigg(\frac{n}{m}\bigg)=f\bigg(n.\frac{1}{m}\bigg)=h_f(n)f\bigg(\frac{1}{m}\bigg)+h_f\bigg(\frac{1}{m}\bigg)f(n)
		$$
		Since by the lemma \ref{inverse} we have : $	f\big(\frac{1}{m}\big)=\frac{-f(m)}{h^2_f(m)}$ , and $h_f\big(\frac{1}{m}\big)=\frac{1}{h_f(m)}$ , then 
		$$   
		f\bigg(\frac{n}{m}\bigg)=\frac{f(n)}{h_f(m)}-\frac{h_f(n)f(m)}{h^2_f(m)}=\frac{f(n)h_f(m)-f(m)h_f(n)}{h^2_f(m)}
		$$
	\end{proof} 
	for all positive integers $n$ and $m$ , The theorem \ref{quotient} may
	be considered a generalized Leibniz rule in the set of rational number $\mathbb{Q}$. This terminology arises from the observation that the arithmetic derivative is L-additive with $h_\delta=n$ ; it satisfies the usual Leibniz rule of quotient :
	$$\delta(\frac{n}{m})=
	\frac{\delta(n)h_\delta(m)-\delta(m)h_\delta(n)}{h^2_\delta(m)} 
	=\frac{m\delta(n)-n\delta(m)}{m^2} 
	$$
	Further, all completely additive functions $f$ are L-additive with $h_f(n) = 1$ , then we can extended any completely addtive function to the set of rational number $\mathbb{Q}$ by this formula :
	$$f\left(\frac{n}{m} \right) =f(n)-f(m)$$ 
	
	For example, the logarithmic derivative of $n$ is completely additive, then we have :
	$$ ld\left(\frac{n}{m} \right)=ld(n)-ld(m)$$
	

	\section{L-additive functions in terms of the Dirichlet convolution}
	
	Above we have seen that many fundamental properties of the extanded of the L-additive function to the set of rational number. We complete this article by changing our point of view slightly and demonstrate that L-additive functions can also be studied in terms of the Dirichlet convolutions by using the theorem \ref{quotient}. 
	
	\medskip
	
	Let $f$ and $g$ be arithmetic functions. Their {\it Dirichlet convolution} is
	$$
	(f\ast g)(n)=\sum_{\substack{a,b=1\\ab=n}}^n f(a)g(b)=\sum_{\substack{d|n}}^n f(d)g\left(\frac{n}{d} \right) .
	$$
	
	where the sum extends over all positive divisors $d$ of $n$ , or equivalently over all distinct pairs $(a, b)$ of positive integers whose product is $n$.\\
	In particular, we have $(f*g)(1)=f(1)g(1)$ ,$(f*g)(p)=f(1)g(p)+f(p)g(1)$ for any prime $p$ and for any power prime $p^m$ we have :
	\begin{equation}
		(f*g)(p^m)=\sum \limits_{j=0}^m f(p^j)g(p^{m-j})
	\end{equation} 
	This product occurs naturally in the study of Dirichlet series such as the Riemann zeta function. It describes the multiplication of two Dirichlet series in terms of their coefficients: 
	\begin{equation}\label{eq:5}
		\bigg(\sum \limits_{n\geq 1}\frac{\big(f*g\big)(n)}{n^s}\bigg)=\bigg(\sum \limits_{n\geq 1}\frac{f(n)}{n^s} \bigg)
		\bigg( \sum \limits_{n\geq 1}\frac{g(n)}{n^s} \bigg)
	\end{equation}
	with Riemann zeta function or  is defined by : $$\zeta(s)= \sum \limits_{n\geq 1} \frac{1}{n^s}$$
	These functions are widely studied in the literature (see, e.g., \cite{book1, book2, book3}).\\

	We let $f(u\ast v)$ denote the product function of $f$ and $u\ast v$, i.e.,
	$$
	(f(u\ast v))(n) = f(n)(u\ast v)(n).
	$$
	
	\begin{theorem}\label{ad-conv}
		An arithmetic function $f$ is completely additive if and only if
		$$
		f(u\ast v) = (fu)\ast v+u\ast (fv)
		$$
		for all arithmetic functions $u$ and $v$.
	\end{theorem}
	
	\begin{proof} 
		(see, e.g., \cite[Proposition 2]{Sc}). 
	\end{proof} 
	
	Next theorems shows the dirichlet convolution of arithmetic function with L-additive functions. 
	
	\begin{theorem}\label{main-the}
		Lets $f$ and $g$ be two arithmetics functions. If $f$ is L-additive and $h_f$ is nonzero-valued, then :
		\begin{eqnarray}
			(f\ast g)(n) = \frac{f(n)}{h_f(n)}(h_f\ast g)(n)-\left( h_f\ast \frac{fg}{h_f}\right) (n)
		\end{eqnarray}
	\end{theorem}
	
	\begin{proof}
		Let $f$ and $g$ be two arithmetic functions. If $f$ is L-additive and $h_f$ is
		nonzero-valued, then applying  the theorem \ref{quotient} on $f\left(\frac{n}{d} \right) $ we have :
		\begin{align*}
			(g*f)(n) &=\sum \limits_{d|n} g(d)f\big(\frac{n}{d}\big)
			=\sum \limits_{d|n} g(d)\bigg( \frac{h_f(d)f(n)-h_f(n)f(d)}{h_f^2(d)}\bigg)
			\\
			&=\sum \limits_{d|n} g(d)\bigg( \frac{f(n)}{h_f(d)}-\frac{h_f(n)f(d)}{h_f^2(d)}\bigg) 
			\\
			&=f(n)\sum \limits_{d|n}\frac{g(d)}{h_f(d)}-h_f(n)\sum \limits_{d|n}\frac{f(d)f(d)}{h_f^2(d)}
			\\
			& =f(n)\left(1 \ast \frac{g}{h_f}\right)(n)-h_f(n)\left(1 \ast \frac{f.g}{h_f^2} \right)(n)
			\\
			&=\frac{f(n)}{h_f(n)}\left(h_f \ast g \right)(n) -\left(h_f \ast \frac{f.g}{h_f} \right)(n)
		\end{align*}
	\end{proof}
	We can prove this formula using the theorem  (\ref{ad-conv}) , since the arithmetic function $\frac{f}{h_f}$ is completely additive by the theorem (\ref{the-1-1}) , so we have :
	$$
	\frac{f}{h_f}\left(h_f \ast g \right)=\left(\frac{f}{h_f}h_f \ast g \right) +\left(h_f \ast \frac{f}{h_f}g \right) =\left(f \ast g \right) +\left(h_f\ast \frac{fg}{h_f} \right) 
	$$ 
	
\begin{corollary}\label{cor-2-1}
	If $f$ is L-additive and $h_f$ is nonzero-valued, then :
	$$
	f\ast \mu h_f=-h_f\ast\mu f
	$$
\end{corollary}
\begin{proof}
	We substitute $g=\mu h_f$ into (\ref{main-the}) gives : 
	$$(f\ast \mu h_f)(n) = \frac{f(n)}{h_f(n)}(h_f\ast \mu h_f)(n)-\left( h_f\ast \frac{f\mu h_f}{h_f}\right) (n)$$
	Since : $$\frac{f(n)}{h_f(n)}(h_f\ast \mu h_f)(n)=f(n)\left(1\ast\mu \right)f(n)\epsilon(n)=0 $$ then we have :
	$$
	(f\ast \mu h_f)(n)=-\left(h_f\ast\mu f \right)(n)
	$$
\end{proof}

\medskip
\noindent
As we are aware, the arithmetic derivative $\delta$ is an L-additive function with $h_{\delta}(n)=Id(n)=n$ , then by using the theorem (\ref{main-the}) we have this corollary :
	\begin{corollary}\label{cor-2-2}
		Given an arithmetic function $g$, then for every positive integer not null $n$  we have :  
		\begin{equation}
			(\delta \ast g)(n)=\frac{\delta(n)}{n}\bigg(Id\ast g\bigg)(n)-\bigg(Id*\frac{g.\delta}{Id}\bigg)(n)
		\end{equation}
	\end{corollary}
	
	\begin{proof} 
		It suffices to notice that $h_{\delta}(n)=Id(n)=n$. 
	\end{proof} 
	\bigskip 
	Now, taking $g(n)=Id(n)=n$, then  the corollary (\ref{cor-2-2}) becomes :
	$$
	(Id\ast \delta)(n)= \frac{\delta(n)}{n}\bigg(Id*Id\bigg)(n)-\bigg(Id*\frac{Id.\delta}{Id}\bigg)(n) 
	=\delta(n)\tau(n)-\left(Id \ast \delta \right)(n)
	$$
	So we have this proven formula in the paper (see, e.g., \cite[Proposition 6]{naoui}) and (see, e.g., \cite[Proposition 2]{Sc}) :
	\begin{equation}\label{equ:1}
		(Id\ast \delta)(n)=\frac{1}{2}\tau(n)\delta(n)	
	\end{equation}
	where $\tau(n)$ is the divisor-number-function .
	
	\bigskip
	We know that $1\ast Id=\sigma $ where $\sigma(n)$ is the sum of the (positive) divisors of $n$ and  $1(n)=1$ for all positive integers $n$ , then by the equality \ref{equ:1} we have :
	\begin{corollary}\label{equ3}
		For every integer $n$ not null  we have :
		\begin{equation}
			\big(\sigma*\delta\big)(n)=\frac{1}{2} \big(1*\tau.\delta\big)(n)
		\end{equation}
	\end{corollary}
	
	\begin{corollary}\label{cor-2-3}
		For every positive integer $n$ not null  we have :
		\begin{equation}
			\delta(n)=\frac{1}{2} \big(Id.\mu*\tau.\delta\big)(n)
		\end{equation}
	\end{corollary}
	\begin{proof}
		Let $n$ a positive integer not null , and $\mu$ the Mobius function .\\
		Since $$\left( Id\ast \delta\right)(n)=\frac{1}{2}\tau(n)\delta(n)$$   
		And :
		$$\left( Id.\mu\ast Id\right)(n)=\epsilon(n)$$  
		then we have : 
		$$
		Id.\mu\ast\left(  Id\ast \delta\right) (n)=\left( Id.\mu\ast \frac{\tau\delta}{2}\right)(n)
		$$
		So :
		$$
		\delta(n)=\frac{1}{2} \big(Id.\mu*\tau.\delta\big)(n) \;\;\;\; since\;\; \left(\epsilon\ast \delta \right)(n)=\delta(n) 
		$$
		\bigskip
		where $\epsilon$ is the multiplicative identity  $\left( \epsilon(n)=\lfloor \frac{1}{n} \rfloor\right) $
	\end{proof}
	\begin{corollary}\label{cor-2-4}
		For every integer $n$ not null  we have :
		\begin{equation}
			\big(Id*Id.\delta\big)(n)=\sigma(n)\delta(n)-\big(Id^2*\delta\big)(n)
		\end{equation}
	\end{corollary}
	
	\bigskip
	\begin{proof}
		For $g(n)=1(n)$, where $1(n)=1$ for all positive integers $n$, this reads by corollary (\ref{cor-2-2})  :
		$$
		(1\ast \delta)(n)=\frac{\sigma(n)\delta(n)-(Id^2 \ast \delta)(n)}{n}  
		$$ 
		The corollary (\ref{cor-2-4}) is satisfied by multiplying the previous equality by id .
	\end{proof}
	\begin{corollary}\label{cor-2-6}
		For every integer $n$ not null  we have :
		\begin{equation}
			\big(Id.\mu\ast\delta\big)(n)=-\big(Id\ast \mu.\delta\big)(n)
		\end{equation}
	\end{corollary}	
	\bigskip
	\begin{proof}
		For $g(n)=Id(n)\mu(n)$ and  for all positive integers $n$, this reads by corollary (\ref{cor-2-2})  :
		$$
		(\delta \ast Id.\mu)(n)=\frac{\delta(n)}{n}\big(Id\ast Id.\mu\big)(n)-\bigg(Id*\frac{Id.\mu.\delta}{Id}\bigg)(n)
		$$ 
		Then : 
		$$
		(\delta \ast Id.\mu)(n)=\delta(n)\big(1\ast \mu\big)(n)-\big(Id\ast \mu.\delta\big)(n)
		$$ 
		Since we know that :
		$$\big(1\ast \mu\big)(n)=\epsilon(n)$$  
		Therefore : 
		$$
		(\delta \ast Id.\mu)(n)=\delta(n)\epsilon(n)-\big(Id\ast \mu.\delta\big)(n)=-\big(Id\ast \mu.\delta\big)(n)
		$$ 
		Because $\delta(n)\epsilon(n)=0$ for all positive integers $n$ . 
	\end{proof}
	We can prove this corollary just by substitute $f(n)=\delta(n)$ with $h_{\delta}(n)=n$ into the corollary (\ref{cor-2-1}) 
	\begin{corollary}\label{cor-2-7}
		For every integer $n$ not null  we have :
		\begin{equation}
			\big(Id.\phi\ast\delta\big)(n)=n\delta(n)-\big(Id\ast \phi.\delta\big)(n)
		\end{equation}
	\end{corollary}
	
	\bigskip
	\begin{proof}
		For $g(n)=Id(n)\phi(n)$ and  for all positive integers $n$, this reads by corollary (\ref{cor-2-2})  :
		$$
		(\delta \ast Id.\phi)(n)=\frac{\delta(n)}{n}\big(Id\ast Id.\phi\big)(n)-\bigg(Id*\frac{Id.\phi.\delta}{Id}\bigg)(n)
		$$ 
		Then : 
		$$
		(\delta \ast Id.\phi)(n)=\delta(n)\big(1\ast \phi\big)(n)-\big(Id\ast \phi.\delta\big)(n)
		$$ 
		Since we know that (see\cite{cat-serie}) : $$\big(1\ast \phi\big)(n)=Id(n)=n$$  Therefore : 
		$$
		(\delta \ast Id.\phi)(n)=n\delta(n)-\big(Id\ast \phi.\delta\big)(n)
		$$ 
	\end{proof}
	
	On the other hand, An arithmetic function $f$ completely additive is L-additive function with $h_{f}(n)=1(n)=1$ for every integer not null, then we have this corollary :
	\begin{corollary} 
		Lets $f$ and $g$ be two arithmetic functions. If $f$ is completely additive, then  for every positive integer not null $n$ by using the theorem \ref{quotient} we have :  
		\begin{equation}
			(f\ast g)(n) = f(n)(1\ast g)(n)-\left( 1\ast fg\right) (n) 
		\end{equation}
	\end{corollary}	
	\bigskip
	as we knows , the derivative arithmetic function $ld$ is completely additive with $h_{ld}(n)=1(n)$ then , for every arithmetic function $f$ we have this result :
	\begin{equation}\label{equ-20}
		(ld\ast f)(n) = ld(n)(1\ast f)(n)-\left( 1\ast ld.f\right) (n) 
	\end{equation}\label{equ-21}
	If multiplie both sides of the previous equality by $Id$ we get this formula :
	\begin{equation}
		\left(\delta\ast Id.f \right)(n)=\delta(n)\left(1\ast f \right)(n)-\left(Id\ast f\delta \right)(n)
	\end{equation}

	 In the same way we can give many result about The prime omega function $\Omega$ and the function $log$ .
	\medskip
\section{Main Results : The generalized von mangoldt function using L-additive function}
	Let $f$ L-additive function with $h_f$ is nonzero-valued, then  Now we defined the von Mangoldt function  related to The function $f$ by :
\begin{equation}\label{def-mangol}
	\Lambda_f (n)={\begin{cases}\frac{f(p)}{h_f(p)}&{\text{if }}n=p^{k}{\text{ for some prime }}p{\text{ and integer }}k\geq 1,\\0&{\text{otherwise.}}\end{cases}}
\end{equation}
then we have this result :
\begin{theorem}\label{the-3-1}
	If $n\geq1$ then we have : 
	$$
	f(n)=h_f(n)\sum_{\substack{d|n}} \Lambda_f(d)
	$$
	That mean by using dirichlet convolution : $f=h_f\ast h_f\Lambda_f$
\end{theorem}
\begin{proof}
	If $n = p_1^{\alpha _1}\ldots  p_s^{\alpha _s}$ then we have :
	$$
	\sum_{\substack{d|n}} \Lambda_f(d)=\sum_{\substack{i=1}}^s \sum_{\substack{k=1}}^i \Lambda_f(p_i^k)=\sum_{\substack{i=1}}^s \sum_{\substack{k=1}}^i \frac{f(p_i)}{h_f(p_i)}=
	\sum_{\substack{i=1}}^s \frac{if(p_i)}{h_f(p_i)}=\frac{f(n)}{h_f(n)}
	$$
	as claimed \\
\end{proof}
\begin{theorem}
	for every positive integer $n$ we have : 
	$$
	{\displaystyle \Lambda_f (n)
		=\sum _{d\mid n}\frac{\mu \left(\frac{n}{d} \right)f\left(d \right) }{h_f\left(d\right)}}
	=-\sum_{\substack{d|n}}\frac{\mu(d)f(d)}{h_f(d)}
	$$
	That is we have $\Lambda_f=\mu\ast \frac{f}{h_f}=-1\ast\frac{\mu f}{h_f}$
\end{theorem}
\begin{proof}
	By Theorem (\ref{the-3-1}) applying Mobius inversion, we have $\Lambda_f=\mu\ast \frac{f}{h_f}$ , Note that :
	\begin{align*}
		\sum _{d\mid n}\frac{\mu \left(\frac{n}{d} \right)f\left(d \right) }{h_f\left(d\right)} &=
		\sum _{d\mid n} \mu(d)\frac{f\left(\frac{n}{d} \right)}{h_f\left(\frac{n}{d} \right)}
		=\sum \limits_{d|n} \mu(d)\frac{h_f(d)}{h_f(n)}
		\bigg( \frac{h_f(d)f(n)-h_f(n)f(d)}{h_f^2(d)}\bigg)
		\\
		&=\sum \limits_{d|n} \frac{\mu(d)h^2_f(d)f(n)}{h_f(n)h^2_f(d)}
		-\frac{\mu(d)h_f(n)h_f(d)f(d)}{h_f(n)h^2_f(d)}
		\\
		&=\frac{f(n)}{h_f(n)}\sum \limits_{d|n} \mu(d)
		-\sum \limits_{d|n}\frac{\mu(d)f(d)}{h_f(d)}
		\\
		& = \frac{f(n)\epsilon(n)}{h_f(n)}-\sum \limits_{d|n}\frac{\mu(d)f(d)}{h_f(d)}
		\\
		&=-\sum \limits_{d|n}\frac{\mu(d)f(d)}{h_f(d)}
	\end{align*}
	Hence $\Lambda_f (n)=-1\ast\frac{\mu f}{h_f}$. This completes the proof .
\end{proof}

\begin{corollary}
	Let $f$ an arithmetic function. If $f$ is L-additive and $h_f$ is nonzero-valued, then :
	\begin{eqnarray}
		(\tau\ast \Lambda_f)(n) =\frac{f(n)\tau(n)}{2h_f(n)} 
	\end{eqnarray}
\end{corollary}
\begin{proof}
	Substituting $g(n)=h_f(n)$ into the theorem (\ref{main-the}) give :
	$$
		(f\ast h_f)(n) = \frac{f(n)}{h_f(n)}(h_f\ast h_f)(n)-\left( h_f\ast \frac{fh_f}{h_f}\right) (n)
	$$ then :
	$$
	(f\ast h_f)(n) =\frac{1}{2}f(n)\tau(n)
	$$
	Since we have that from the theorem (\ref{the-3-1}) :
	$$
	f(n)=\left( h_f\ast h_f\Lambda_f\right) (n)
	$$
Then we find that :
$$
\left(h_f\ast f \right)(n)= \left(h_f\ast  h_f\ast h_f\Lambda_f \right)(n)=
h_f(n)\left(1\ast 1\ast\Lambda_f \right)(n) 
$$
We conclude that :
$$
h_f(n)\left(\tau\ast\Lambda_f \right)(n) =\frac{1}{2}\tau(n)f(n)
$$
This completes the proof of Theorem
\end{proof}

As we knows the arithmetic function $f$ completely additive is L-additive function with $h_{f}(n)=1(n)=1$ for every integer not null, then we have the von Mangoldt function  related to The function $f$ defined by :
\begin{equation}\label{equ-21}
	\Lambda_f (n)={\begin{cases}f(p)&{\text{if }}n=p^{k}{\text{ for some prime }}p{\text{ and integer }}k\geq 1,\\0&{\text{otherwise.}}\end{cases}}
\end{equation}

\begin{corollary}\label{cor-3-8}
Let $f$ an arithmetic function completely additive, then : 
	$$
	f(n)=\sum_{\substack{d|n}} \Lambda_f(d)
	$$
	That mean by using dirichlet convolution : $f=1\ast \Lambda_f$
\end{corollary}
\begin{corollary}\label{cor-3-9}
Let $f$ an arithmetic function completely additive, then we have: 
	$$
	{\displaystyle \Lambda_f (n)
		=\sum _{d\mid n}\mu \left(\frac{n}{d} \right)f\left(d \right)
		=-\sum_{\substack{d|n}}\mu(d)f(d)}
	$$
	That is we have $\Lambda_f=\mu\ast f=-1\ast \mu f$
\end{corollary}
The definition (\ref{def-mangol}) may be considered a generalized von mangoldt function. This terminology arises from the observation that the logarithm is L-additive with $h_{log}(n)=1$; it satisfies the usual von mangoldt function denoted by $\Lambda$
$$
\Lambda(n)=\Lambda_{log} (n)={\begin{cases}\frac{log(p)}{h_{log}(p)}=log(p) &{\text{if }}n=p^{k}{\text{ for some prime }}p{\text{ and integer }}k\geq 1,\\0&{\text{otherwise.}}\end{cases}}
$$
By using the corollary (\ref{cor-3-8}) and (\ref{cor-3-9}) the Properties of the von mangoldt function is hold and we have 
$$
log=1\ast \Lambda \;\;\;and\;\;\; \Lambda=\mu\ast\log=-1\ast\mu\log
$$
Now we can defined the von mangoldt function associed to arithmetic derivative $ld$ by : 
$$
\Lambda_{ld} (n)={\begin{cases}\frac{1}{p} &{\text{if }}n=p^{k}{\text{ for some prime }}p{\text{ and integer }}k\geq 1,\\0&{\text{otherwise.}}\end{cases}}
$$
Substituting $h_{ld}(n)=1$ into the corollary (\ref{cor-3-8}) and (\ref{cor-3-9}) gives :
\begin{equation}\label{von-ld}
	ld(n)=\left( 1\ast \Lambda_{ld}\right)(n)
\end{equation}
And : 
\begin{equation}
	\Lambda_{ld}(n)=\left( \mu\ast ld\right) (n)=-\left( 1\ast ld\mu\right)(n) 
\end{equation}

For later convenience we introduce the prime function denoted by $F$ defined by :
$$
F(s)=\sum \limits_{p}\frac{1}{p^{s+1}-p}
$$
and note that it converges for $Re(s)>0$ It is an analog of the Riemann zeta function, described in (see, e.g., \cite{prime}), with the sum taken over prime numbers instead of all natural numbers , then we have this result about the von mangoldt function associed to arithmetic derivative 
\begin{lemma}
	Let $s$ a number complex such that $Re(s)>0$ , then we have :
$$
\sum \limits_{n\geq 1} \frac{\Lambda_{ld}(n)}{n^s}=F(s)
$$	
\end{lemma}
\begin{proof}
		Let $s$ a number complex such that $Re(s)>0$ , then 
		\begin{align*}
			\sum \limits_{n\geq 1} \frac{\Lambda_{ld}(n)}{n^s} &=\frac{\Lambda_{ld}(1)}{1^s}+\frac{\Lambda_{ld}(2)}{2^s}+\frac{\Lambda_{ld}(3)}{3^s}+\frac{\Lambda_{ld}(4)}{4^s}+\frac{\Lambda_{ld}(5)}{5^s}+\ldots +\frac{\Lambda_{ld}(16)}{16^s}+\ldots
			\\
			&=\frac{1}{2^{s+1}}+\frac{1}{3^{s+1}}+\frac{1}{2^{2s+1}}+\frac{1}{5^{s+1}}+\frac{1}{7^{s+1}}+\frac{1}{2^{3s+1}}+\ldots+\frac{1}{2^{4s+1}}+\ldots
			\\
			&=\sum \limits_{p}\sum \limits_{k\geq 1}\frac{1}{p^{ks+1}}= \sum \limits_{p}\frac{1}{p}\sum \limits_{k\geq 1}\frac{1}{p^{ks}}
			\\
			& =\sum \limits_{p}\frac{1}{p}\sum \limits_{k\geq 1}\bigg(\frac{1}{p^{s}} \bigg)^k=
			\sum \limits_{p}\frac{1}{p}.\frac{1}{p^s}.\frac{1}{1-\frac{1}{p^s}}
			\\
			&=\sum \limits_{p}\frac{1}{p^{s+1}-p}
		\end{align*}
	which completes the proof
\end{proof}

\begin{theorem}\label{the-3-3}
Let $s$ a number complex such that $Re(s)>0$ , then we have :
$$
\sum \limits_{n\geq 1} \frac{\delta(n)}{n^s}=\zeta(s-1)F(s-1)
$$	
\end{theorem}
\begin{proof}
	By Using $\left( \ref{von-ld}\right)$  we have : 
	$$\delta(n)=\left(Id\ast Id.\Lambda_{ld} \right)(n)$$ 
	 then by the formula $\left(\ref{eq:5} \right) $ we have  :
	$$
	\sum \limits_{n\geq 1} \frac{\delta(n)}{n^s}=\sum \limits_{n\geq 1} \frac{\left(Id\ast Id.\Lambda_{ld} \right)(n)}{n^s}=
	\zeta(s-1)\sum \limits_{n\geq 1} \frac{\Lambda_{ld}(n)}{n^{s-1}}=\zeta(s-1)F(s-1)
	$$ 
\end{proof}
\bigskip
\begin{corollary}
	for every $s\in\mathbb{C}$ where $Re(s)>2$ we have : 
	$$
	\sum \limits_{n\geq 1} \frac{\tau(n)\delta(n)}{n^s}=2\zeta^2(s-1)F(s-1) 
	$$
\end{corollary}
\begin{proof}
	If we apply the relation (\ref{eq:5}) on the equality (\ref{eq:1}) ,then for every complex number $s$ we get that :
	
	$$
	\sum \limits_{n\geq 1} \frac{\tau(n)\delta(n)}{n^s}=2\zeta(s-1)\sum \limits_{n\geq 1} \frac{\delta(n)}{n^s} \;\;\;\;\;\; \left( for\;\;Re(s)>2\right) 
	$$
	then by the theorem (\ref{dir-der}) for every $s\in\mathbb{C}$ where $Re(s)>2$ we have : 
	$$
	\sum \limits_{n\geq 1} \frac{\tau(n)\delta(n)}{n^s}=2\zeta^2(s-1)F(s-1) 
	$$
\end{proof}

\begin{corollary}
		for every $s\in\mathbb{C}$ where $Re(s)>2$ we have : 
	$$
	\sum \limits_{n\geq 1} \frac{\mu(n)\delta(n)}{n^s}
	=\frac{-F(s-1)}{\zeta(s-1)}
	$$
\end{corollary}
in the same way , If we apply the relation (\ref{eq:5}) to the corollary (\ref{cor-2-6}) we have that :
$$
\zeta(s-1)\sum \limits_{n\geq 1} \frac{\mu(n)\delta(n)}{n^s}
=-\sum \limits_{n\geq 1} \frac{\mu(n)}{n^{s-1}}
\sum \limits_{n\geq 1}\frac{\delta(n)}{n^{s}}\;\;\;\;\;\; \left( for\;\;Re(s)>2\right) 
$$
Since (see, e.g., \cite{cat-serie}) :
$$
\sum \limits_{n\geq 1} \frac{\mu(n)}{n^s}=\frac{1}{\zeta(s)}
$$
Therefore by using the result of theorem (\ref{the-3-3}) we have that :
$$
\sum \limits_{n\geq 1} \frac{\mu(n)\delta(n)}{n^s}
=\frac{-F(s-1)}{\zeta(s-1)}
\;\;\;\;\;\; \left( for\;\;Re(s)>2\right) 
$$
\begin{corollary}
	for every $s\in\mathbb{C}$ where $Re(s)>3$ we have
$$
\sum \limits_{n\geq 1} \frac{\phi(n)\delta(n)}{n^s}
=\frac{\zeta(s-2)}{\zeta(s-1)}\bigg( F(s-2)-F(s-1)\bigg) 
$$
\end{corollary}
\begin{proof}

We know from $\left(\cite{cat-serie} \right) $ that :
$$\sum \limits_{n\geq 1} \frac{\phi(n)}{n^s} =\frac{\zeta(s-1)}{\zeta(s)}$$
And by corollary (\ref{cor-2-7}) we have :
$$\big(Id.\phi\ast\delta\big)(n)=n\delta(n)-\big(Id\ast \phi.\delta\big)(n)$$
then : 
$$
\sum \limits_{n\geq 1} \frac{\phi(n)\delta(n)}{n^s}
=\frac{1}{\zeta(s-1)}
\sum \limits_{n\geq 1}\frac{\delta(n)}{n^{s-1}}-\frac{\zeta(s-2)}{\zeta^2(s-1)}\sum \limits_{n\geq 1}\frac{\delta(n)}{n^{s}}\;\;\;\;\;\; \left( for\;\;Re(s)>3\right) 
$$
\end{proof}

On the other hand as a consequence of the corollary (\ref{cor-2-4}) we have : $$\big(Id*Id.\delta\big)(n)=\sigma(n)\delta(n)-\big(Id^2*\delta\big)(n)$$ then by apply the formula \ref{eq:5} and the result of the theorem \ref{the-3-3} we have:
\begin{corollary}
		for every $s\in\mathbb{C}$ where $Re(s)>3$ we have :
	$$
	\sum \limits_{n\geq 1} \frac{\sigma(n)\delta(n)}{n^s}
	=\zeta(s-1)\zeta(s-2)\left( F(s-2)+F(s-1)\right) 
	$$
\end{corollary}

	By using this equality let $f(n)=Id_k(n)$ then we have 
\begin{corollary}
	Let $k\in\mathbb{N}$ , then	for every $s\in\mathbb{C}$ where $Re(s)>k+2$ we have :
	$$
	\sum \limits_{n\geq 1} \frac{\sigma_k(n)\delta(n)}{n^s}
	=\zeta(s-1)\zeta(s-k-1)\left( F(s-1)+F(s-k-1)\right) 
	$$
\end{corollary}
\begin{proof}
		Let $k\in\mathbb{N}$ ,applying equality (\ref{equ-21}) to $f(n)=Id_k(n)$  we obtain :
		$$
			\left(\delta\ast Id.Id_k \right)(n)=\delta(n)\left(1\ast Id_k \right)(n)-\left(Id\ast Id_k\delta \right)(n)
		$$
		Since :
		$$
		\left(1\ast Id_k \right)(n)=\sigma_k(n)
		$$
		then we have : 
		$$
		\left(\delta\ast Id_{k+1}\right)(n)=\delta(n)\sigma_k(n)-\left(Id\ast Id_k\delta \right)(n)
		$$
		Now applying equality (\ref{eq:5}) we obtain:
		$$
		\sum \limits_{n\geq 1} \frac{\left(\delta\ast Id_{k+1}\right)(n)}{n^s}=
		\sum \limits_{n\geq 1} \frac{\delta(n)\sigma_k(n)}{n^s}-
		\sum \limits_{n\geq 1} \frac{\left(Id\ast Id_k\delta \right)(n)}{n^s}
		$$
		Therefore,
		$$
		\sum \limits_{n\geq 1} \frac{\delta(n)\sigma_k(n)}{n^s}=
		\zeta(s-k-1)\sum \limits_{n\geq 1} \frac{\delta(n)}{n^s}+
		\zeta(s-1)\sum \limits_{n\geq 1} \frac{\delta(n)}{n^{s-k}}
		$$
		Since by the theorem (\ref{the-3-3}) we have:
		$$
		\sum \limits_{n\geq 1} \frac{\delta(n)}{n^s}=\zeta(s-1)F(s-1)
		$$
		then :
		$$
			\sum \limits_{n\geq 1} \frac{\delta(n)\sigma_k(n)}{n^s}=
			\zeta(s-k-1)\zeta(s-1)F(s-1)+\zeta(s-1)\zeta(s-k-1)F(s-k-1)
		$$
		and the proof is complete.
\end{proof}

\medskip
the aim of this study is to giving some result that may help us to find a relation defined by Dirichlet product which makes it possible to calculate the series of Dirichlet $
\sum \limits_{n\geq 1} \frac{f(n)\delta(n)}{n^s}
$ for many arithmetic function $f$.

\section{Conclusion :}
		The von Mangoldt function $\Lambda_f$ related to the L-additive function $f$ is the best way to be solved us the problem of  the Dirichlet series of the  arithmetic derivative  
			

\begin{thebibliography}{33}
			
			\bibitem{Htd}
			V. Ufnarovski and B. Åhlander.
			\newblock  How to differentiate a number,
			\newblock {\em J . Integer Seq., 6 }, (2003).
			
			\bibitem{Ufna}
			P. Haukkanen, J. K. Merikoski, and T. Tossavainen.
			\newblock   The arithmetic derivative and Leibniz additive functions,
			\newblock {\em Notes Number Theory Discrete Math., 24.3 }, (2018).
			
			\bibitem{book1} T.~M.~Apostol,  {\it Introduction to Analytic Number
				Theory}, Springer-Verlag, New York (1976). 
			
			\bibitem{Ba} E.~J.~Barbeau, Remarks on an arithmetic derivative,
			{\it Canad. Math. Bull.}~{\bf 4}(2), 117--122 (1961).
			
			\bibitem{Ch} L.~M.~Chawla, A note on distributive arithmetical functions,
			{\it J. Nat. Sci. Math.}~{\bf 13}(1), 11--17 (1973).
			
			\bibitem{book2}
			M.Lewinter and J.Meyer.
			\newblock Elementary  number theory with programming ,
			\newblock {\em Wiley  }, (2016).
			
			\bibitem{book3}
			Heng Huat Chan,
			\newblock Analytic number theory for undergraduates ,
			\newblock {\em World Scientific Publishing Co. Pte. Ltd.  }, (2009).
			
			\bibitem{Ha}
			P. ~Haukkanen, A note on generalized multiplicative and generalized additive arithmetic functions, {\it  Math. Stud.} {61}, 113--116 (1992). 
			
			\bibitem{HMT} P.~Haukkanen, J.~K.~Merikoski, T.~Tossavainen, On arithmetic
			partial differential equations, {\it J. Integer Seq.}~{19}, {Article~16.8.6} (2016).
			
			\bibitem{naoui} Es-said En-naoui, Dirichlet product of derivative arithmetic with an arithmetic function multiplicative, {\it https://arxiv.org/abs/1908.07345.} (2019).
			
			\bibitem{KM}
			I.~Kiuchi, M.~Minamide, On the Dirichlet convolution of completely additive functions. {\it J. Integer Seq.} {17}, Article 14.8.7 (2014). 
			
			\bibitem{Ko} J.~Kovi\v c, The arithmetic derivative and antiderivative, {\it J. Integer Seq.}~{15}, {Article~12.3.8} (2012).
			
			\bibitem{LT} 
			V.~Laohakosol, P.~ Tangsupphathawat, Characterizations of additive functions. 
			{\it Lith. Math. J.} {\bf 56}(4), 518--528 (2016).  
			
			\bibitem{Mc} P.~J. McCarthy, {\it Introduction to Arithmetical
				Functions}, Springer-Verlag, New York (1986). 
			
			\bibitem{Sc} E.~D.~Schwab, Dirichlet product and completely additive arithmetical functions, {\it Nieuw Arch. Wisk.} {\bf 13}(2), 187--193 (1995). 
			
			\bibitem{Sh} H.~N.~Shapiro, {\it Introduction to the Theory of Numbers},  Wiley InterScience, New York (1983).
			
			\bibitem{UA} V.~Ufnarovski, B.~\AA hlander, How to differentiate a number, {\it J. Integer Seq.} {6},  {Article 03.3.4} (2003).
			
			\bibitem{ext} Merikoski, Jorma K.; Haukkanen, Pentti; Tossavainen, Timo , Complete additivity, complete multiplicativity, and Leibniz-additivity on rationals {\it  Integers.} {} (2021).
			
			\bibitem{cat-serie} Gould, H. W. and Shonhiwa, T. ,A catalog of interesting Dirichlet series {\it  Missouri J. Math.
				Sci., } {1:2–18} (2008).
			
			\bibitem{prime} NIST Digital Library of Mathematical Functions. http://dlmf.nist.gov, Release 1.0.9 of, 
			 {\it  Online companion to.}\cite{next}  (2014-08-29).
			 
			 \bibitem{next} Olver, F. W. J. and Lozier, D. W. and Boisvert, R. F. and Clark, C. W.,. http://dlmf.nist.gov, Release 1.0.9 of,  {\it  Handbook of Mathematical Functions.} Cambridge University Press, New York, NY, . Print companion to\cite{prime}  (2010).
			 
		\end{thebibliography}
	\end{document}